\documentclass[a4paper, 11 pt, twoside]{amsart}
\usepackage{srollens-en}[2011/03/03]
\usepackage[left = 3.5cm, right = 3.5cm, headsep = 6mm,footskip = 10mm, top = 35mm, bottom = 35mm, footnotesep=5mm, headheight =2cm]{geometry}
\usepackage[expansion=false]{microtype}
\usepackage{booktabs, multirow,  tabularx}
\usepackage[usenames, dvipsnames]{xcolor}

\usepackage{tikz, tikz-cd}
\usepackage{enumitem}
\setlist[description]{leftmargin=0cm,  labelindent=\parindent}

\usepackage[unicode,bookmarks, pdftex, backref = page]{hyperref}
\hypersetup{colorlinks=true,citecolor=NavyBlue,linkcolor=NavyBlue,urlcolor=Orange, pdfpagemode=UseNone, breaklinks=true}

\newcommand{\Hbar}[1]{\bar\gothH^{\left(#1\right)}}
\newcommand{\Hcan}[1]{\gothH^{\left(#1\right)}}

\title{Standard stable Horikawa surfaces}

\author[J. Rana]{Julie Rana}
 \address{Julie Rana\\Department of Mathematics, Lawrence University, 711 E. Boldt Way, Appleton WI 54911, USA.}
\email{ranaj@lawrence.edu}

 \author[S. Rollenske]{S\"onke Rollenske}
 \address{S\"onke Rollenske\\FB 12/Mathematik und Informatik\\
 Philipps-Univer\-si\-t\"at Marburg\\
 Hans-Meerwein-Str. 6\\
 35032 Marburg\\
 Germany}
 \email{rollenske@mathematik.uni-marburg.de}

\date{\today}

\begin{document}
\begin{abstract}
 We consider  the stable compactification $\bar \gothH$ of the moduli space of Horikawa surfaces with $K_X^2 = 2p_g(X) -4$.
 
 When $K_X^2 =8\ell$ we show that the closures of  the two components $\gothH^{\mathrm I}$ and $\gothH^{\mathrm {II}}$ of the Gieseker moduli space intersect, for $\ell>2$  in a divisor parametrising explicitly described semi-smooth surfaces.
 
 With growing $K_X^2$ we find an increasing number of generically non-reduced irreducible components in the same connected component of the moduli space of stable surfaces.
\end{abstract}

\maketitle
\setcounter{tocdepth}{1}
\tableofcontents

\section{Introduction}
For complex minimal surfaces of general type we have the classical Noether inequality 
$ K_X^2 \geq 2p_g -4$, see \cite[Ch. VII]{BHPV}.
In 1976 Horikawa classified surfaces satisfying equality in the first of a series of seminal papers \cite{horikawa1}; in his honor these surfaces are now called Horikawa surfaces.

Let us denote by $\gothH_{2k}$ the Gieseker moduli space of Horikawa surfaces with $K_X^2 = 2k$. Then Horikawa showed that $\gothH_{2k}$ is irreducible unless $2k\equiv 0 \mod 8$ in which case $\gothH_{2k} = \gothH^{\mathrm I}_{2k}\sqcup \gothH^{\mathrm{II}}_{2k}$ has two connected components of the same dimension. For $2k\equiv 8 \mod 16$ he showed that the two components parametrise non-diffeomorphic surfaces, but ever since it has remained open whether $\gothH^{\mathrm I}_{16\ell}$ and $\gothH^{\mathrm{II}}_{16\ell}$ parametrise diffeomorphic surfaces \cite{A06, FS97, LP11}.

Nowadays, the Gieseker moduli space $\gothH_{2k}$ embeds into a natural compactification $\bar\gothH_{2k}$, the moduli space of stable Horikawa surfaces, which parametrises stable surfaces with the same Hilbert polynomial (see \cite{kollar12, KollarModuli}). The starting point of the present work was the question of whether the closures $\bar\gothH^{\mathrm I}_{16\ell}$ and $\bar \gothH^{\mathrm{II}}_{16\ell}$ intersect inside $\bar \gothH_{16\ell}$.
\begin{custom}[Theorem A]
The intersection  $\bar\gothH^{\mathrm I}_{8\ell}$ and $\bar \gothH^{\mathrm{II}}_{8\ell}$
for $(\ell>1)$ 
contains a divisor $\gothD$ parametrising explicitly described non-normal (but semi-smooth) surfaces. The intersection is not normal crossing at the general point of $\gothD$. 
\end{custom}
We will prove this result in Theorem \ref{thm: connection} and Corollary \ref{cor: Thm A local structure}. 
 Horikawa surfaces  with $K_X^2 = 8$ are a bit of an exception because not all such surfaces are double covers of Hirzebruch surfaces.  We show in Section \ref{sect: connection 8}  that $\bar\gothH^{\mathrm I}_{8}$ and $\bar \gothH^{\mathrm{II}}_{8}$ intersect but do not have as good a control over the intersection locus. 

While Theorem A gives us an explicit description how to move from one component to the other, the standard tools in 4-manifold topology do not seem to be able to control the resulting surgery, so that the answer to the diffeomorphism question posed above remains elusive for now. Some further remarks on the case $K_X^2 = 16$ can be found in \cite{RR22b}.

The ingredients in the proof of Theorem A are some abstract deformation theory and explicit toric constructions. The latter is a lower-dimensional version of the scrollar deformations used by Coughlan and Pignatelli to study canonical threefolds of general type on the (3-dimensional) Noether line \cite{cp22}. It is remarkable that also in the case of threefolds every eighth instance of the moduli space has two irreducible components. In contrast to Theorem A, these components intersect in a locus parametrising  threefolds with canonical singularities.

The methods used to prove Theorem A
lead us to consider more general stable Horikawa surfaces. Surprisingly, we find with growing $K_X^2$ a tail of trailing irreducible components in the moduli space.
\begin{custom}[Theorem B]
Let $k\geq 5$.
The connected component of $\bar \gothH_{2k}$ containing classical Horikawa surfaces contains
\[ \bar\gothH_{2k}\supset \gothH_{2k} \cup\bigcup_{ \substack{k>m>\frac{k+4}{2}\\ m\equiv k \textrm{ mod } 2}} \Hbar{m}_{2k},\] 
where the $\Hbar{m}_{2k}$ are generically non-reduced, irreducible components of dimension $5k+4m+19> \dim \gothH_{2k}$. 
\end{custom}
We illustrate the phenomenon schematically in Figures \ref{fig: H26} and \ref{fig: H32}, the proof of Theorem B can be found in Section \ref{sect: proof B}. 
\begin{figure}
 \caption{Standard components in $\bar \gothH_{26}$}\label{fig: H26}
 \begin{tikzpicture}[thick]
   \draw[draw = red] (0,0) to[bend left] node[above]{ $ \gothH_{26}$}(2,0);
 \draw[draw = blue] (2,0) to ++ (0, 2 ) to[bend right] node[below] {$\Hbar{9}_{26}$} ++ (2,0) to ++(0,-2) to[bend left] cycle;
  \draw[draw = Plum] (4,2)  to ++ (1, 0.5) to[bend right] ++ (2,0) to ++(-1,-0.5);
  \draw[Plum] (6,0)  to ++ (1, 0.5) to ++ (0,2);
\draw[draw = Plum, dashed] (4,0) to ++ (1, 0.5) to[bend right] ++ (2,0);
\draw[draw = Plum, dashed] (4,0) ++ (1, 0.5) to++ (0,2);
 \draw[draw = Plum] (4,0) to ++ (0, 2 ) to[bend right] node[below, fill = white] {$\Hbar{11}_{26}$} ++ (2,0) to ++(0,-2) to[bend left] cycle; \end{tikzpicture}
\end{figure}
\begin{figure}
 \caption{Standard components in $\bar \gothH_{32}$}\label{fig: H32}
 \begin{tikzpicture}[thick]
   \draw[draw = red] (-2,0) to[bend left] node[above]{ $\bar \gothH^{\mathrm I}_{32}$}++ (2,0) node[below, color = black] {$\gothD$};
      \draw[draw = Orange] (0,0) to[bend left] node[above]{ $\bar \gothH^{\mathrm{II}}_{32}$}++(2,0);
 \draw[draw = blue] (2,0) to ++ (0, 2 ) to[bend right] node[below] {$\Hbar{12}_{32}$} ++ (2,0) to ++(0,-2) to[bend left] cycle;
 \draw[draw = Plum] (4,2)  to ++ (1, 0.5) to[bend right] ++ (2,0) to ++(-1,-0.5);
  \draw[draw = Plum] (6,0)  to ++ (1, 0.5) to ++ (0,2);
\draw[draw = Plum, dashed] (4,0) to ++ (1, 0.5) to[bend right] ++ (2,0);
\draw[draw = Plum, dashed] (4,0) ++ (1, 0.5) to++ (0,2);
 \draw[draw = Plum] (4,0) to ++ (0, 2 ) to[bend right] node[below,  fill = white] {$\Hbar{14}_{32}$} ++ (2,0) to ++(0,-2) to[bend left] cycle;
  \end{tikzpicture}
\end{figure}

\subsection*{Acknowledgments}
We would like to thank Jonny Evans, whose conversation and engaging blog post \cite{evans-blog} originally sparked our interest in this problem. 

Roberto Pignatelli and Stephen Coughlan introduced the second author to the scrollar deformations of \cite{cp22}, which simplified previous computations tremendously.  Some abstract and concrete aspects of deformation theory were discussed with Donatella Iacono and Hans-Christian von Bothmer. We are also grateful to Enrico Schlesinger for pointing out a useful reference.

The first author is partially supported by NSF LEAPS-MPS grant \#2137577. She would like to thank the second author for his time and support during a visit, funded by an AWM Mentoring Travel Grant, to Universit\"at Marburg.

 The second author is grateful for support of the DFG (Project number 509274422). He  would like to thank the first author and her family for the hospitality in Kathmandu, where some of the results were first conceived.

\section{Standard stable  Horikawa surfaces } 
We work over the complex numbers. General references for the notions related to stable surfaces and their moduli are \cite{kollar12, KollarSMMP, KollarModuli}. 
All necessary information about double covers in this context can be found in \cite{alexeev-pardini12}.

Let $\IF_m$ be a Hirzebruch surface. We denote the negative curve by $\sigma_\infty$, so $\sigma_\infty^2 = -m$, and the class of a fibre by $\Gamma$. We also fix a section disjoint from $\sigma_\infty$, namely $\sigma_0 \in |\sigma_\infty + m \Gamma|$; in the toric model introduced later, $\sigma_0$ can be chosen to be invariant. 
 \begin{defin}
  A standard stable Horikawa surface of type $(m)$ is a double cover 
  \[ f\colon X \to \IF_m\]
  branched over $B \in |6\sigma_\infty + 2a\Gamma|$
such that $X$ has slc singularities and $K_X$ is ample.
We call it a classical Horikawa surface if $X$ has at most canonical singularities. 
  \end{defin}
\begin{lem}\label{lem: standard}
 A standard stable Horikawa surface of type $(m)$ exists for $a>2m+2$ and satisfies
 \[K_X^2  = 4a-6m-8, \qquad p_g(X) = 2a-3m-2,\]
 so $K_X^2  =  2p_g(X)-4$.
In addition:
\begin{enumerate}
 \item If $2a\geq 6m$ then the linear system has no base points and the general branch divisor is smooth and connected.
 \item If $6m>2a\geq 5m$ then $B = \sigma_\infty + B'$ and $B'$ moves in a base-point free linear system with $\sigma_\infty . B' = 2a-5m$.

In particular, the general branch divisor is smooth and disconnected for $2a = 5m$.
  \item If $5m >2a > 4m+4$ then the general branch divisor is $B = 2 \sigma_\infty + B'$ with $B'$ in the base-point free linear system $ |4 \sigma_0 + 2(a-2m)\Gamma|$. 
 In this case, the general  $X$ is non-normal with normal crossing singularities at the general point of $\inverse f (\sigma_\infty)$ and   $2(a-2m)$ pinch points.
 \end{enumerate}
 In particular, a classical Horikawa surface of type $(m)$ with $K_X^2 = 2k$ exists if and only if $m\leq\frac{k+4}2$ and $m$ and $k$ have the same parity. 
\end{lem}
  \begin{rem}\label{rem: which cases}
   Note that for $m = 0, 1, 2, 3$ only the first case can occur, while for $m = 4, 5, 6$ only the first two cases can occur. For $m\geq 7$ all three cases are possible. 
   
   Note also that $K_X^2$ is always even, and it is divisible by $4$ if and only if the type $(m)$ is even. 
     \end{rem}

 \begin{proof}
   The canonical divisor of $X$ is 
 \begin{align*}
 K_X &=  f^*\left(K_{\IF_{m}} + \frac 12 B\right)\\
   & = f^*\left(-2\sigma_\infty - (m+2)\Gamma + 3\sigma_\infty + a \Gamma\right)\\
    & = f^*\left(  \sigma_\infty + (a-m-2) \Gamma\right),
 \end{align*}
 and this bundle is ample if and only if it is positive on the pullback of $\sigma_\infty$ if and only if $a>2m+2$.
Then
 \[K_X^2 = 2\left(  \sigma_\infty + (a-m-2) \Gamma\right)^2 = -2m + 4 (a -m -2) = 4a-6m -8\]
and 
\[p_g(X) = h^0(\sigma_\infty + (a-m-2)\Gamma) = h^0(\ko_{\IP^1}(a-m-2)) + h^0(\ko_{\IP^1}(a-m-2))  = 2a-3m -2. 
\]
The rest of the claims rely on a standard computation on Hirzebruch surfaces, determining how often a particular linear system has to contain the negative section.
  \end{proof}

We denote the subset of the moduli space of stable Horikawa surfaces parametrising standard stable Horikawa surfaces with $K_X^2 = 2k$  of type $(m)$ by $\Hbar{m}_{2k}\subset\bar \gothH_{2k}$ and the subset of classical Horikawa surfaces of type $(m)$ by $\Hcan{m}_{2k} = \Hbar{m}_{2k} \cap \gothH_{2k}$.

The moduli space of stable surfaces also carries a natural scheme structure and if $\Hbar{m}_{2k}$ forms an open subset of an  irreducible component of $\bar\gothH_{2k}$ then we consider it with this scheme structure. Otherwise, we just consider it as a reduced subscheme of $\bar\gothH_{2k}$.

\begin{prop}\label{prop: dim of strata}
 The moduli spaces $\Hbar{m}_{2k}$ of standard stable Horikawa surfaces of type $(m)$ are irreducible and 
 \[ \dim \Hbar{m}_{2k} =
 \begin{cases}
 7k+28 & m=0\\
 7k+29-m & \frac{k+4}3 \geq m >0\\
      6k+2m +24 & \frac{k+4}2 \geq m > \frac{k+4}{3}\\
      5 k + 4m +19 &k > m > \frac{k+4}2\\
    \end{cases}.
\]
\end{prop}
\begin{proof} Since the Picard group of a Hirzebruch surface does not contain $2$-torsion, the double cover is determined by its branch divisor. The condition that a double cover has semi-log-canonical singularities is open, compare \cite[Lemma 3.2]{anthes20}. Therefore, a complete family of standard Horikawa surfaces of type $(m)$ is parametrised by an open subset of the linear system $|6\sigma_\infty + 2a\Gamma|$ on $\IF_m$ where $2k = 4a-6m-8$,  so its image in the moduli space is irreducible as well. 

 The dimension of the linear system is a straightforward cohomology computation:
 \begin{align*}
 h^0(\IF_{m}, 6\sigma_\infty + 2a\Gamma) & =\sum_{i = 0}^6 h^0(\IP^1, \ko_{\IP^1}(-im + 2a))\\
  &= \sum_{i = 0}^4 (-im + 2a+1)+ \sum_{i = 5}^6 h^0(\IP^1, \ko_{\IP^1}(-im + 2a))\\
  \intertext{ and since   $2a>4m+4$ and $2a = k+3m+4$}
  &=(10a-10m+5)+ h^0(\ko_{\IP^1}(2a-5m))+h^0(\ko_{\IP^1}(2a-6m))\\
  &= \begin{cases}
      14a - 21 m +7 & 2a\geq 6m\\
      12a-15m + 6 & 6m >2a \geq 5m \\
      10 a -10m +5 & 5m > 2a>4m+4
     \end{cases}\\
  &= \begin{cases}
      7k+35 & k \geq 3m-4\\
      6k+3m +30 & 3m-4>k\geq 2m-4 \\
      5 k + 5m +25 &2m-4> k > m
     \end{cases}.
  \end{align*}
Then, because $h_0(\kt_{\IF_{m}}) = m +5$  for $m\geq 1$ and $h^0(\kt_{\IP^1\times \IP^1}) = 6$, computed in   \cite{horikawa1} or \cite[Appendix B]{Sernesi}, we have 
\[ \dim \Hbar{m}_{2k}  = \dim |B| - \dim \Aut \IF_m 
  =  
 \begin{cases}
 7k+28 & m=0\\
 7k+29-m & \frac{k+4}3 \geq m >0\\
      6k+2m +24 & \frac{k+4}2 \geq m > \frac{k+4}{3}\\
      5 k + 4m +19 &k > m > \frac{k+4}2\\
    \end{cases}
\]
as claimed. 
\end{proof}

With the above notation we can phrase some of Horikawa's original results as follows, see also \cite[VII.9]{BHPV}.
\begin{thm}[Horikawa \cite{horikawa1}]\label{thm: classical stuff}
Let $\gothH_{2k}$ be the moduli space of (classical) Horikawa surfaces with $K_X^2 = 2k = 2p_g(X) - 4$ for $k\neq 1, 4$.
\begin{enumerate}
\item If $k$ is odd, then 
\[ \gothH_{2k}  = \bigcup_{d = 0}^{\lfloor \frac{k+2}{4}\rfloor}\Hcan{2d+1}_{2k}\]
 is irreducible of dimension $7k+28$.
  \item If $k$ is even and  $2k \not \equiv 0 \mod 8$ then
 \[ \gothH_{2k}  = \bigcup_{d= 0}^{\lfloor \frac{k+4}{4}\rfloor}\Hcan{2d}_{2k}\]
 is irreducible of dimension $7k+28$.
  \item If $2k \equiv 0 \mod 8$ then $\gothH_{2k} = \gothH^{\mathrm I}_{2k} \sqcup \gothH^{\mathrm{II}}_{2k}$ has two connected and irreducible components, both of dimension $7k+28$: the general component
  \[ \gothH^{\mathrm I}_{2k} =  \bigcup_{d =0 }^{\frac{k}{4}}\Hcan{2d}_{2k}\]
  and the special component
  \[ \gothH^{\mathrm{II}}_{2k} = \Hcan{\frac{k+4}{2}}_{2k}.\]
  If $2k \equiv 8 \mod 16$ then smooth surfaces in the respective components are not diffeomorphic.   
 \end{enumerate}
\end{thm}
 Horikawa surfaces with $K_X^2 = 2$ and and $p_g  (X) = 3$ are double covers of the projective plane branched over a sufficiently general octic and some information on their stable degenerations can be found in \cite{anthes20}. 
The case $K_X^2 = 8$ will be discussed briefly in Section \ref{sect: connection 8}.

In the following we want to investigate how the subsets $\Hbar{m}_{2k}$ interact inside $\bar\gothH_{2k}$. We focus particularly on  the closures $\bar \gothH^{\mathrm I}$ and $\bar \gothH^{\mathrm{II}}$ of the special and general components in the cases where $K_X^2$ is divisible by $8$. 
As a byproduct, we will actually reprove most of Theorem \ref{thm: classical stuff}.

For later use we also note the following.
\begin{cor}\label{cor: increasing dimension trailing.}
 Fix $K_X^2 = 2k>8$. Then the dimensions of the non-classical subsets $\Hbar{m}_{2k}$ are strictly increasing:
 \begin{enumerate}
  \item If $k$ is odd then
  \[ \dim \gothH_{2k}< \dim \Hbar{2\lfloor \frac{k+2}{4}\rfloor+3}_{2k} < \dots < \dim \Hbar{k-2}_{2k}.\]
    \item If $k$ is even then
  \[ \dim \gothH_{2k}< \dim \Hbar{2\lfloor \frac{k+4}{4}\rfloor+2}_{2k} < \dots < \dim \Hbar{k-2}_{2k}.\]
 \end{enumerate}
\end{cor}
\begin{proof}
 From Lemma \ref{lem: standard} one can check that the listed spaces are  exactly the ones containing no classical Horikawa surfaces. The rest follows by comparing their dimensions computed in Proposition \ref{prop: dim of strata} with $\dim \gothH_{2k} = \dim \Hbar{0}_{2k}$.
\end{proof}

  \section{Connecting \texorpdfstring{ $\gothH^{\mathrm I}$ and $\gothH^{\mathrm {II}}$}{the general and special component}}\label{sect: conection 1}

Recall, e.g. from \cite[Sect. 5.2]{cls2012},  that for any integer $\alpha$, the $\mathbb{Z}^2$-graded ring, with variables and weights
\[ \begin{pmatrix}
    t_0 & t_1 & x_0 & x_1 \\
    1& 1 & \alpha & \alpha-m\\
    0& 0& 1& 1
   \end{pmatrix}
\]
and irrelevant ideal $(t_0, t_1)\cap (x_0, x_1)$, is the Cox ring of the Hirzebruch surface $\mathbb{F}_m$.
The negative section $\sigma_\infty$ is given by $\{x_1=0\}$, the positive section $\sigma_0$ by $\{x_0=0\}$, and the fibers by vanishing of linear polynomials $\{f_1(t_0,t_1)=0\}$.

\subsection{Horikawa surfaces in weighted projective bundles}\label{sec: toric rep}

The fibration on $\mathbb{F}_m$ induces a pencil of genus two curves on any standard stable Horikawa surface $X$ of type $(m)$. Since a genus two curve can be canonically embedded in $\mathbb{P}(1,1,3)$, we can thus describe such a surface $X$ with $K_X^2=2k$ as a hypersurface in a toric variety $T_{m,k}$, which is a $\mathbb{P}(1,1,3)$ bundle over $\mathbb{P}^1$ that varies depending on $m$ and $k$. 

In this section, we suppose that $k$ is even (and therefore $m$ as well by Lemma~\ref{lem: standard}), and so $K_X^2\equiv 0$ mod 4. We will treat the case that $k$ is odd in Section \ref{sect: connection odd}. To simplify exposition, we let $m=2d$, and $k=2n$. Then the surfaces we want to describe are hypersurfaces in the toric threefold $T_{2d,2n}$ (defined for $0\le 2d \le n+2$) given by 
 \[ \begin{pmatrix}
    t_0 & t_1 & x_0 & x_1 &  z\\
    1& 1 & d-n-2 & -d-n-2 &  -2(n+2)\\
    0& 0& 1& 1&  3
   \end{pmatrix}
\]
with irrelevant ideal $(t_0, t_1)\cap (x_0, x_1, z)$. The surface $X$ arises as a sufficiently general hypersurface of bidegree $\mat{ -4(n+2)\\ 6}$, and as such is defined by a polynomial $z^2+f(x_0,x_1,t_0, t_1)$, where we eliminate the linear term in $z$ by completing the square. Only the first entry of the degree vector is relevant to determine which monomials appear in  $f(x_0,x_1,t_0, t_1)$ ; denoting it by $\deg_1$ we have for example $\deg_1(z^2)=-4(n+2)$. 

Let us consider three examples corresponding to the lowest and highest possible values of $d$, in terms of $n$:
\begin{exam}\label{ex: general}
 If $d = 0 $ then the matrix of weights becomes
  \[ \begin{pmatrix}
    t_0 & t_1 & x_0 & x_1 & z\\
    1& 1 & -n-2 &-n-2 & -2(n+2)\\
    0& 0& 1& 1& 3
   \end{pmatrix}
\]
In this case, the  monomials appearing in $f$ are of the form  $x_0^ax_1^{6-a} g_{2(n+2)}(t_0, t_1)$, which are bihomogeneous when considered in the usual grading. We thus recognize $X$ as a double cover of $\mathbb{F}_0=\mathbb{P}^1\times\mathbb{P}^1$.
\end{exam}
The case of particular interest is the following, which for a choice of coefficients gives a key example of a singular stable Horikawa surface in $\bar\gothH_{2k}^\mathrm{II}$.

\begin{exam}\label{exam: special}
If $n$ is even and $d = \frac{n+2}{2}$, then the  weight matrix is 
 \[ \begin{pmatrix}
    t_0 & t_1 & x_0 & x_1 & z\\
    1& 1 & -(\frac{n}{2}+1)&-\frac{3}{2}(n+2)& -2(n+2)\\
    0& 0& 1& 1& 3
   \end{pmatrix}
\]
and we see that $ \deg_1 x_0^6 = -3n-6> \deg_1 z^2 =\deg _1 x_0^5x_1$. Since multiplying with a polynomial in $t_0, t_1$ can only increase the degree, a general polynomial in the linear system is of the form 
\begin{equation}\label{eq: special smooth}  z^2 - x_1(\mu x_0^5 +\dots), 
 \end{equation}
where $\mu$ is a nonzero constant. The branch divisor then contains $\sigma_\infty = \{ x_1 = 0 \}$ once and the rest is disjoint from $\sigma_{\infty}$ because the  term $(\mu x_0^5 +\dots)$ has $x_0^5$ with non-vanishing coefficient.

If we eliminate the variable $z$, the remaining matrix describes the Hirzebruch surface $\IF_{n+2}$, and so stable hypersurfaces given by such equations lie in the special component $\bar \gothH_{4n}^\mathrm{II}$.
\end{exam}

\begin{exam}\label{exam: divisor D}
In the same ambient toric threefold and linear system  as in Example \ref{exam: special} we now consider the hyperplane  of surfaces defined by an equation as in \eqref{eq: special smooth} where the coefficient $\mu$  vanishes, that is the polynomial is of the form 
\begin{equation}\label{eq of divisor} 
z^2 - x_1^2(x_0^4g_{n+2}(t_0, t_1) +\dots).
 \end{equation}
The general such  surface is stable: it is a double cover of $\IF_{n+2}$  branched over $2\sigma_\infty+B'$ where $B'$ is smooth and intersects the negative section transversally. Thus $X$ is smooth outside the non-normal locus, which is the pullback of the negative section. Over the general point of the negative section, $X$ has normal crossing singularities, locally $x_1^2 - z^2$, and at the $n+2$ intersection points of $\sigma_\infty $ and $B'$ the surface $X$ has pinch points, i.e., locally $z^2-tx_1^2$.

Together, these surfaces form a divisor $\gothD$ in $\Hbar{n+2}_{4n} \subset \bar \gothH_{4n}^\mathrm{II}$.
\end{exam}

\subsection{Connecting the general and the special component ($K_X^2>8$)}\label{sect: connection general special}
We assume in this section that $K_X^2=2k=4n\equiv 0 \mod 8$ and $n\geq 3$. Then by Theorem \ref{thm: classical stuff}  the Gieseker moduli space $\gothH_{4n}$ is the union of the two irreducible and connected components of dimension $14(n+2)$ containing general, respectively special surfaces, $\gothH_{4n} = \gothH_{4n}^{\mathrm I} \sqcup \gothH_{4n}^{\mathrm{II}}.$
We show that the closures of these components intersect in the stable compactification $\bar \gothH_{4n}$:
\begin{thm}\label{thm: connection}
If $K_X^2=4n\equiv 0$ mod 8 and $n\geq 3$, then the intersection $\bar\gothH_{4n}^{\mathrm{I}} \cap \bar \gothH_{4n}^{\mathrm {II}}$ contains the divisor $\gothD$ of semi-smooth double covers of $\IF_{n+2}$ described in Example~\ref{exam: divisor D}.
\end{thm}

The general surface in $\gothD$ can be smoothed to a surface in $\gothH_{4n}^{\mathrm{II}}$ using the parameter $\mu$ in Example \ref{exam: special}. So it remains to show that it can be smoothed to the other component as well. 

We start by constructing an explicit family over $\mathbb{A}^1_\lambda$, depending on a choice of two polynomials $p_0,p_1 \in R_{(\frac{n}{2}+1)}$, with special fibre a singular surface described in Example~\ref{exam: divisor D} and general fiber a smooth surface of the type described in Example~\ref{ex: general}.

Consider the toric fourfold $T$ given by 
 \[ \begin{pmatrix}
    t_0 & t_1 & y_0 & x_0 & x_1 & z\\
    1& 1 & -(\frac{n}{2}+1) & -n-2 & -n-2 & -2(n+2)\\
    0& 0& 1& 1& 1& 3
   \end{pmatrix}
\]
with irrelevant ideal $(t_0, t_1)\cap (y_0 , x_0, x_1, z)$.
Let $R =\IC[t_0, t_1]$ with the usual grading.

Inside $T$ let us first consider the family $T_\lambda$ over $\mathbb{A}^1_\lambda$ of threefolds given by 
\begin{equation}
 \lambda y_0 = p_1x_0 - p_0 x_1 \text{ with $p_i \in R_{(\frac{n}{2}+1)}$ and $\gcd(p_0, p_1) =1$.} 
 \end{equation}
If $\lambda\neq 0 $, then we eliminate the variable $y_0$ and obtain $T_{0,2n}$, the threefold considered in Example \ref{ex: general}. 

If $\lambda = 0 $ then we can introduce a new variable 
\[ y_1 = \frac{x_0}{p_0} = \frac{x_1}{p_1}, \quad \deg y_1 =\mat{-\frac{3}{2}(n+2)\\1} ,\]
because the $p_i$ cannot vanish simultaneously. The resulting equations
\begin{equation}
\label{eq; coord sub}
 x_0  = p_0 y_1, 
 \qquad x_1 = p_1 y_1
\end{equation}
eliminate the variables $x_0, x_1$ and we recover $T_{n+2, 2n}$, the toric threefold from Example \ref{exam: special}.

To define a family of hypersurfaces in $T_\lambda$ that restricts in the desired fashion, we need to analyse which elements in the relevant linear system on the central fibre are global on the family. That is, we need to describe all polynomials in $t_0, t_1, x_0, x_1, y_0$ of bidegree $\mat{-4(n+2)\\ 6}$. These are sums of elements of the subspaces 
\begin{equation}
 \label{eq: lin sys general}
 y_0^{6-i}\cdot \Sym^i\langle x_0, x_1\rangle \cdot R_{(i-2)(\frac{n}{2}+1)} \qquad (i\geq 2)
\end{equation}
because these have $\deg_1$ equal to $-(6-i)\cdot (\frac{n}{2}+1)- i \cdot 2(\frac{n}{2}+1) + (i-2)(\frac{n}{2}+1)= -4(n+2)$. Note that the cases $i =0, 1$ do not occur, because there are no polynomials of negative degree in $R$. 

Now we let $\lambda $ go to zero to see which monomials we get on the central fibre. To describe this, let $V = \langle p_0, p_1\rangle$. Then substituting \eqref{eq; coord sub} into \eqref{eq: lin sys general} we obtain on the special fibre all equations of the form $z^2 - f$ with $f$ in the subspace generated by  
\[
y_0^{6-i}y_1^i \cdot\Sym^i V \cdot R_{(i-2)(\frac{n}{2}+1)} \subset y_0^{6-i}y_1^i 
\cdot R_{(2i-2)(\frac{n}{2}+1)}
\qquad (i\geq 2)
.\]

\begin{lem}\label{lem: multiplication maps}
 Let $p_0, p_1\in R_{\frac{n}{2}+1}$ without common divisor and let $i\geq 3$. Then  the multiplication map 
 \[ \phi^{p}_i \colon  \Sym^i V \cdot R_{(i-2)(\frac{n}{2}+1)} \to  R_{(2i-2)(\frac{n}{2}+1)}\]
 is surjective.
\end{lem}
\begin{proof}
Let $A = A' = \IP^1$. We have maps
 \[ 
  \begin{tikzcd}[column sep = large]
   A'\rar {\pi = (p_0:p_1) }& A \rar{ v_i} & \IP^i 
  \end{tikzcd}
 \]
where $v_i$ is the Veronese embedding of $\mathbb{P}^1$ into $\mathbb{P}^i$ as a rational normal curve.  Note that the $\ko_A$-algebra  $\pi_*\ko_{A'}$ is a torsion-free coherent sheaf on $\IP^1$, hence a vector bundle. The trace map splits off a trivial summand, and by Grothendieck's Lemma (see e.g.  \cite[Cor. 5.2.8]{Huybrechts}) the whole bundle is a direct sum of line bundles.  Since the cohomology of $\pi_*\ko_{A'}$ gives the cohomology of $\ko_{A'}$ we have 
\[ \pi_*\ko_{A'} = \ko_A \oplus \ko_A(-1)^{\oplus \frac{n}{2}}\]
 We then interpret 
 \[\Sym ^i V = \pi^*H^0(A, \ko_A(i)) = \pi^*v_i^*H^0(\IP^i, \ko_{\IP^i}(1))\text{ and }R_{\alpha(\frac{n}{2}+1)} = H^0(A', \pi^*\ko_{A}(\alpha)).\]
Restricting the Euler sequence of $\IP^i$ to the rational normal curve and pulling back to $A'$ we get
\[ 0 \to \pi^*v_i^*\left(\Omega_{\IP^i}(1)\right)\tensor \pi^*\ko_A(i-2) \to \Sym^i V \tensor \pi^*\ko_A(i-2) \to \pi^*\ko_A(2i-2)\to 0 \]
and find by taking global sections that the cokernel of the multiplication map is 
\begin{align*}
 \coker \phi^{p}_i & = H^1\left(A', \pi^*v_i^*\left(\Omega_{\IP^i}(1)\right)\tensor \pi^*\ko_A(i-2) \right)\\
 & = H^1\left(A, v_i^*\left(\Omega_{\IP^i}(1)\right)\tensor \ko_A(i-2) \tensor \pi_*\ko_{A'} \right)\\
 & =  H^1\left(A, v_i^*\left(\Omega_{\IP^i}\right)\tensor \ko_A(2i-2) \tensor \left(\ko_{A}\oplus \ko_A(-1)^{\oplus \frac{n}{2}} \right)\right),
 \end{align*}
which by \cite[Proposition 5A.2]{Bayer-Eisenbud} 
becomes
\begin{align*}
\coker \phi^{p}_i & =  H^1\left(A, \left(\ko_A(-i-1)^{\oplus i}\right)\tensor \ko_A(2i-2) \tensor \left(\ko_{A}\oplus \ko_A(-1)^{\oplus \frac{n}{2}} \right)\right)\\
 & =  H^1\left(A, \ko_A(i-3)^{\oplus i}\right)\oplus H^1\left(A,  \ko_A(i-4) ^{\oplus i\frac{n}{2}} \right)\\ 
 & = 0 \qquad (\text{for $i\geq 3$}).
\end{align*}
Note that for $i = 2$ it is quite obvious that the map is not onto.\end{proof}

Theorem \ref{thm: connection} now follows immediately from the slightly more precise result.

\begin{prop}
 Let $X_0$ be a hypersurface as in Example \ref{exam: divisor D} such that in equation (\ref{eq of divisor}), the polynomial $g_{2(\frac{n}{2}+1)} = p_0p_1$ where $p_i\in R_{\frac{n}{2}+1}$ and $\gcd(p_0,p_1) = 1$. Then $X_0$ can be smoothed to a Horikawa surface in $\gothH_{2n}^{\mathrm{I}}$.
\end{prop}
\begin{proof}
By assumption and Lemma \ref{lem: multiplication maps} we can choose monomials such that the surface defined in the toric fourfold $T$ by the equations
 \[  \lambda y_0 = p_1x_0 - p_0 x_1 \text{ and } z^2 + y_0^4x_0x_1 + (\text{terms not divisible by $y_0^4$})=0 \]
 degenerates over $\mathbb{A}^1_\lambda$ to $X_0$ in the family described above. After possibly adding $\lambda$ times a general polynomial, the general fibre will be smooth as claimed. 
\end{proof}

\subsection{Connecting $\gothH^\mathrm{I}_8$ and $\gothH^\mathrm{II}_8$}
\label{sect: connection 8}
Here we consider the last remaining case, Horikawa surfaces with $K_X^2 = 8$ and $p_g(X) = 6$. There are two new constructions here:
\begin{description}
 \item[Type $(\infty)$] $X$ is a double cover of the projective plane branched over a smooth curve of degree ten. 
 A quick dimension count shows that this family $\Hcan{\infty}_8$ has dimension $\dim |\ko_{\IP^2}(10)| - \dim \Aut(\IP^2) = 57$. 
 \item[Type $(4')$] $X$ is a double cover of the cone over a rational normal curve of degree four branched over the vertex and a quintic section.
 Thus we can realise $X$ as a (sufficiently general) hypersurface of degree $20$ in weighted projective space $\IP(1,1,4,10)$. 
 
 The main difference between this and the general case is that $K_{\IF_4} + \frac 12\left( 6\sigma_\infty + 20 \Gamma\right) $ is not ample on $\IF_4$  but only big and nef and contracts the negative section. So the double cover of $\IF_4$ branched over a smooth  $B\in |6 \sigma_\infty + 20 \Gamma|$ gives a smooth minimal Horikawa surface, but the preimage of the negative section is a $-2$ curve, which we need to contract to get the canonical model.

 Arguing as in Proposition \ref{prop: dim of strata} the family $\Hcan{4'}_8$ is of dimension $56$. 
 \end{description}

By \cite[Section 4]{horikawa1} again the moduli space has two connected components, namely,
  \[ \gothH^{\mathrm I}_{8} =  \Hcan{0}_{8}\sqcup \Hcan{2}_{8}
  \text{ and } \gothH^\mathrm {II}_{8} =  \Hcan{\infty}_{8}\sqcup \Hcan{4'}_{8}.
  \]
The main point is that every surface of type $(4')$ deforms to one of type $(\infty)$ by taking a suitable  double cover of the $\IQ$-Gorenstein smoothing of $\IP(1,1,4)$ to $\IP^2$.

  Note that in this case the dimensions of the two components are different, namely $\dim \gothH^{\mathrm I}_{8}=56$ and $\dim \gothH^\mathrm {II}_{8}=57$.

We will now prove:
\begin{prop}\label{prop: intersection 8}
 The closures $\bar\gothH^{\mathrm I}_{8}$ and $ \bar  \gothH^\mathrm {II}_{8}$ intersect in $\bar \gothH_8$. 
\end{prop}

\begin{rem}
 Arguing more carefully as in Section \ref{sect: connection general special} one can be a bit more precise: The intersection of the two 56-dimensional subsets $\Hbar{0}_{8}$ and $\Hbar{4'}_{8}$ 
 contains a subset $\gothD$ of dimension $55$,
 which is therefore a divisor in $\bar\gothH_8^\mathrm{I}$ and a subset of codimension two in $\bar\gothH_8^\mathrm{II}$. We illustrate this in Figure \ref{fig: H8}.
\end{rem}
\begin{figure}
 \caption{Standard strata in $\bar \gothH_{8}$}\label{fig: H8}
 \begin{tikzpicture}[thick]
 \draw[draw = blue] (2,0) node[left]{$\bar\gothH_8^{\mathrm{I}}$} to ++ (0, 2 )  to[bend right] node[below] {$\Hbar{0}_{8}$} ++ (2,0) to ++(0,-2) node[below]{ $\gothD$}to[bend left] cycle;
  \filldraw[fill = Plum!10!white, draw = Plum] (4,0) to ++ (0, 2 ) to[bend right]  ++ (2,0) to ++(0,-2) to[bend left] cycle;
  \draw[Plum] (4,2)  to ++ (1, 0.5) to[bend right] ++ (2,0) to ++(-1,-0.5) ;
  \draw[draw =Plum] (6,0)  to ++ (1, 0.5)  to node [left] {$\Hbar{\infty}_{8}$} ++ (0,2);
\draw[Plum, dashed] (4,0) to ++ (1, 0.5) to[bend right] ++ (2,0);
\draw[Plum, dashed] (4,0) ++ (1, 0.5) to++ (0,2) ;
 \draw[draw=Plum] (4,0) to ++ (0, 2 ) to[bend right] node[below, fill = Plum!10!white] {$\Hbar{4'}_{8}$} ++ (2,0) 
  to ++(0,-2) 
   to[bend left] cycle; 
 \node at (7, 0) {$\bar\gothH_8^{\mathrm{II}}$};
 \end{tikzpicture}
\end{figure}

\begin{proof} The aim is to  construct a surface in the closure of both components. In contrast to the previous section, this surface will be normal with one elliptic singularity.

Consider the toric fourfold $T$ given by 
 \[ \begin{pmatrix}
    t_0 & t_1 & y_0 & x_0 & x_1 & z\\
    1& 1 & 4 & 2 & 2 & 10\\
    0& 0& 1& 1& 1& 3
   \end{pmatrix}
\]
with irrelevant ideal $(t_0, t_1)\cap (y_0 , x_0, x_1, z)$.
Let $R =\IC[t_0, t_1]$ with the usual grading. 

Inside $T$ let us consider the family of threefolds $T_\lambda$ over $\mathbb{A}^1_\lambda$ given by 
\[ \lambda y_0 = p_1x_0 - p_0 x_1 \text{ with $p_0, p_1 \in R_{2}$ and  $p_0p_1$ without multiple zeros,} 
\] which
we intersect with a sufficiently general  hypersurface of bidegree $\mat{20\\ 6}$ given by 
\[ z^2 + x_0x_1y_0^4 + \text{lower order terms in $y_0$} + \lambda g(x_0, x_1, t_0, t_1) = 0 \]
to get a family of surfaces $X_\lambda$. 
If $\lambda\neq 0 $, then we eliminate the variable $y_0$ and find, up to a change of basis in the weight matrix,  the threefold $T_\lambda\isom T_{0,4}$ defined at the beginning of Section~\ref{sec: toric rep}. Thus $X_\lambda$ is a Horikawa surface of type $(0)$ for $\lambda \neq 0$.

Now let us consider the central fibre $\lambda = 0$. 
We introduce a new variable 
\[ y_1 = \frac{x_0}{p_0} = \frac{x_1}{p_1}, \quad \deg y_1 =\mat{0\\1} ,\]
because the $p_i$ cannot vanish simultaneously. The resulting equations
$ x_0  = p_0 y_1$, and $ x_1 = p_1 y_1$
eliminate the variables $x_0, x_1$ and we get the toric threefold $T_0$ with weights
 \[ \begin{pmatrix}
    t_0 & t_1 & y_0 & y_1 & z\\
    1& 1 & 4 & 0 & 10\\
    0& 0& 1& 1&  3
   \end{pmatrix},
   \]   
in which the equation for $X_0$ becomes
\[  z^2 + p_0p_1y_0^4y_1^2 + \text{(lower order terms in $y_0$)} = 0.\]
Translating back to geometry, i.e., projecting to the first four coordinates, $\pi\colon X_0\to \IF_4$ is a double cover branched over $2\sigma_\infty + B'$ where $B'$ intersects $\sigma_\infty $ transversally in four points, because $p_0p_1$ has no multiple zeros. 

Note that $K_{X_0} = \pi^* (\sigma_\infty + 4\Gamma)$ is not ample, because it is trivial on the non-normal locus, which is the preimage of $\sigma_\infty$. To understand the (log)-canonical model, let us take the normalisation, 
\[
 \begin{tikzcd}
  \bar X_0 \arrow {rr}{\nu} \arrow{dr}{\bar \pi}[swap]{\text{branched over $B'$}} && X_0 \arrow{dl}{\text{branched over $2\sigma_\infty+B'$}}[swap]{\pi}\\
  & \IF_4 
 \end{tikzcd}.
\]
Then $E = \inverse{\bar\pi} (\sigma_\infty)$ is a smooth elliptic curve with $E^2  = -8$ and $\nu^*K_{X_0} = K_{\bar X_0} + E$. Since all sections of $m (K_{\bar X_0} + E)$ are constant on $E$ we have $\pi^*\colon H^0(mK_{X_0}) \isom  H^0(m(K_{\bar X_0}+E))$ and the log-canonical model of $X_0$ is the log-canonical model of $(\bar X_0, E)$. The result is the surface $Y_0$, where we contract the elliptic curve $E$ to an elliptic singularity.

On the level of the ambient toric threefold $T_0$ this corresponds to eliminating the variable $y_1$, that is, the rational projection 
\[ T_0 \dashrightarrow \IP(1,1,4,10).\]
The resulting hypersurface $Y_0$ can be deformed to a general such hypersurface, which is a Horikawa surface of type $(4')$. Thus the surface $Y_0$ is in the closure of both 
  $\gothH^\mathrm{I}_8$ and $\gothH^\mathrm{II}_8$.
\end{proof}

\section{Connecting adjacent (non-classical) components}\label{sect: conection 2}
In this section, we use the toric construction described above to show that all standard stable Horikawa surfaces are contained in the same connected component of $\bar\gothH_{2k}$. The details of  the proof depend on the parity of $k$.

\subsection{Connecting $\Hbar{2d}_{4n}$ and $\Hbar{2d+2}_{4n}$}
We begin by assuming that $k$ is even, and take $k=2n$.
\begin{prop}\label{prop: connect to previous even}
 The subsets $\Hbar{2d}_{4n}$ and $\Hbar{2d+2}_{4n}$ are in the same connected component of $\bar\gothH_{4n}$.
\end{prop}
\begin{proof}
Note that the closure of the Gieseker moduli space of classical Horikawa surfaces is connected  by Horikawa's Theorem \ref{thm: classical stuff} and by Section \ref{sect: conection 1}, so we may assume that $\Hbar{2d+2}_{4n}$ is a component containing only non-classical surfaces; that is, $n\leq 2d <2n-2$.

Since both subsets are connected, it suffices to exhibit a (Gorenstein hence $\IQ$-Gorenstein) family of standard stable Horikawa  surfaces, where the general fibre is of type $(2d)$ and a special fibre is of type $(2d+2)$.

 Consider for $n\leq 2d <2n-2$ 
 the  toric fourfold $T$ given by 
 \[ \begin{pmatrix}
    t_0 & t_1 & y_0 & x_0 & x_1 & z\\
    1& 1 & d-n-1 & d-n-2 & -d-n-2 & -2(n+2)\\
    0& 0& 1& 1& 1& 3
   \end{pmatrix}
\]
with irrelevant ideal $(t_0, t_1)\cap (y_0 , x_0, x_1, z)$.
Let $R =\IC[t_0, t_1]$ with the usual grading. 

Inside $T$ let us first consider the family of threefolds $T_\lambda$ over $\mathbb{A}^1_\lambda$ given by 
\begin{equation}
 \lambda y_0 = p_1x_0 - p_0 x_1 \text{ with $p_1 \in R_{1}$, $p_0 \in R_{2d+1}$ and  $\gcd(p_0, p_1) =1$.} 
 \end{equation}
If $\lambda\neq 0 $, then we eliminate the variable $y_0$ and find the threefold $T_{2d,2n}$ defined at the beginning of Section~\ref{sec: toric rep}.

If $\lambda = 0 $ then we can introduce a new variable 
\[ y_1 = \frac{x_0}{p_0} = \frac{x_1}{p_1}, \quad \deg y_1 =\mat{-d-n-3\\1} ,\]
because the $p_i$ cannot vanish simultaneously. The resulting equations
\begin{equation}
\label{eq; coord sub with k}
 x_0  = p_0 y_1, 
 \qquad x_1 = p_1 y_1
\end{equation}
eliminate the variables $x_0, x_1$ and we get $T_{2d+2, 2n}$.

To prove that $\bar\gothH^{2d}_{4n}$ meets $\bar\gothH^{2d+2}_{4n}$ we need an equation of bidegree $\mat{-4(n+2)\\6}$ on $T$ of the form $z^2+f(t_0, t_1,y_0, x_0, x_1)$ that defines a stable surface if we set $\lambda = 0 $. 
Writing
\[\mat{-4(n+2)\\6} =  \mat{2n-2d\\0}+ 2 \mat{-d-n-2\\1}  + 4 \mat{d-n-1 \\1}\]
we can set
\[ f (t_0, t_1,y_0, x_0, x_1)= g_{2n-2d}(t_0, t_1)\cdot  x_1^2 \cdot \prod_{i = 1}^4\left( y_0 + a_i x_0+b_ix_1)\right),\]
for general $a_i\in R_{1}$ and $b_i\in R_{2d+1}$.

Upon intersection with $T_\lambda$ for $\lambda=0$, we obtain 
\[ f(t_0, t_1,y_0, p_0y_1, p_1y_1)  = g_{2n-2d}(t_0, t_1)p_1^2  y_1^2 \cdot \prod_{i = 1}^4\left( y_0 + (a_ip_0+ b_i p_1)y_1)\right).\]
For general choices (for example, choosing $p_0=t_0^{2d+1}$, $p_1=t_1$ and generic $a_i$, $b_i$), this equation defines a stable surface, because the branch curve $B$ is the union of $2n-2d$ fibres, twice the negative section, twice a fibre, and four sufficiently general sections in $|\sigma_0|$ that are disjoint from $\sigma_\infty$, thus $(\IF_{2d+2}, \frac 12 B)$ is a log-canonical pair,
compare e.g \cite{alexeev-pardini12}. 

More concretely, the double cover has (generically) $16(d+n+2)$ $A_1$ singularities, normal crossing singularities over the general point of the double locus, $2n-2d +4$ pinch points and a degenerate cusp with local equation $z^2 +x^2t^2$ over the point where $2\sigma_\infty$ meets the double fibre.  
This proves the claim. 
\end{proof}

\subsection{Connecting $\Hbar{2d+1}_{4n-2}$ and $\Hbar{2d-1}_{4n-2}$}\label{sect: connection odd}
We suppose that $k = \frac 12 K_X^2$ is odd, and take $k=2n-1$. Then by Lemma \ref{lem: standard} the type $(m)$ has to be odd as well.

To connect the components of standard stable Horikawa surfaces of odd type, we follow the same strategy  employed above: we realise the individual surfaces as hypersurfaces of a toric threefold and then connect these constructions inside a toric fourfold. 

\begin{exam}\label{ex: toric for k odd}
Consider for $0<d<n$ the toric threefold $T_{2d-1, 2n-1}$ described via
 \[ \begin{pmatrix}
    t_0 & t_1 & x_0 & x_1 & z\\
    1& 1 & d-n-1 & -d-n & -2n\\
    0& 0& 1& 1& 3
   \end{pmatrix}
\]
with irrelevant ideal $(t_0, t_1)\cap ( x_0, x_1, z)$.
A general equation of bidegree $\mat{ -4n\\ 6}$ without linear term in $z$ is of the form
\[ z^ 2 + \sum_{i = 0 }^6g_{2n+6d-i(2d-1)} x_0^ix_1^{6-i}\]
and thus defines a double cover of $\IF_{2d-1}$ branched over a curve  in $|6\sigma_\infty+ (2n+6d) \Gamma|$. By Lemma \ref{lem: standard}, we can describe all standard stable Horikawa surfaces of type $(2d-1)$ with $K_X^2 = 2(2n-1)$ in this way.
\end{exam}

\begin{prop}\label{prop: connect to previous odd}
 The subsets $\Hbar{2d+1}_{4n-2}$ and $\Hbar{2d-1}_{4n-2}$ are in the same connected component of $\bar\gothH_{4n+2}$. 
\end{prop}
\begin{proof}
We exhibit a ($\IQ$-)Gorenstein family of standard stable Horikawa  surfaces, where the general fibre is of type $(2d-1)$ and a special fibre is of type $(2d+1)$. 

 Consider for $n+2\leq 2d <2n+2$
  the  toric fourfold $T$ given by 
 \[ \begin{pmatrix}
    t_0 & t_1 & y_0 & x_0 & x_1 & z\\
    1& 1 & d-n & d-n-1 & -d-n & -2n\\
    0& 0& 1& 1& 1& 3
   \end{pmatrix}
\]
with irrelevant ideal $(t_0, t_1)\cap (y_0 , x_0, x_1, z)$.
Let $R =\IC[t_0, t_1]$ with the usual grading. 

Inside $T$ let us consider the family of threefolds $T_\lambda$ over $\mathbb{A}^1_\lambda$ given by 
\begin{equation}
 \lambda y_0 = p_1x_0 - p_0 x_1 \text{ with $p_1 \in R_{1}$, $p_0 \in R_{2d}$ and  $\gcd(p_0, p_1) =1$.} 
 \end{equation}
If $\lambda\neq 0 $, then we eliminate the variable $y_0$ and find the toric threefold $T_{2d-1, 2n-1}$ from  Example \ref{ex: toric for k odd}.
If $\lambda = 0 $ then we can introduce a new variable 
\[ y_1 = \frac{x_0}{p_0} = \frac{x_1}{p_1}, \quad \deg y_1 =\mat{-d-n-1\\1} ,\]
because the $p_i$ cannot vanish simultaneously. The resulting equations
 \[x_0  = p_0 y_1, 
 \qquad x_1 = p_1 y_1
 \]
eliminate the variables $x_0, x_1$ and we find the threefold $T_{2d+1, 2n-1}$ 
with weights
 \[ \begin{pmatrix}
    t_0 & t_1 & y_0 & y_1 & z\\
    1& 1 & d-n & -d-n-1 & -2n\\
    0& 0& 1& 1& 3
   \end{pmatrix}
\]
To prove that $\bar\gothH^{2d-1}_{4n-2}$ meets $\bar\gothH^{2d+1}_{4n-2}$ we need an equation of bidegree 
$\mat{-4n\\6}$ 
on $T$ of the form $z^2+f(t_0, t_1,y_0, x_0, x_1)$ that defines a stable surface if we set $\lambda = 0 $. 
Writing
\[\mat{-4n\\6} =  \mat{2n -2d\\0}+ 2 \mat{-d-n\\1}  + 4 \mat{d-n \\1}\]
we can set
\[ f (t_0, t_1,y_0, x_0, x_1)= g_{2n-2d}(t_0, t_1)\cdot  x_1^2 \cdot \prod_{i = 1}^4\left( y_0 + a_i x_0+b_ix_1)\right),\]
for general $a_i\in R_{1}$ and $b_i\in R_{2d}$ and $g_{2n-2d}\in R_{2n-2d}$, because $n>d$. 

Upon intersection with $T_\lambda$ for $\lambda=0$, we obtain
\[ f(t_0, t_1,y_0, p_0y_1, p_1y_1)  = g_{2n-2d}(t_0, t_1)p_1^2  y_1^2 \cdot \prod_{i = 1}^4\left( y_0 + (a_ip_0+ b_i p_1)y_1)\right),\]
which defines a stable surface for sufficiently general choices, as in the proof of Proposition \ref{prop: connect to previous even}.
\end{proof}

\section{Infinitesimal deformations of standard stable Horikawa surfaces}

We start by considering some consequences of the general theory of deformations of maps (see e.g. \cite[Appendix C]{GLS}), which we will then apply to Horikawa surfaces. In full generality, the information of $\Def_A$,  deformations of an object $A$, is encoded in some cotangent complex $\IL_A^\bullet$ and the associated cohomology groups $T_A^i$  and sheaves $\kt_A^i$. 
More concretely, for a finite morphism 
\[ f \colon X \to W\]
we consider 
\begin{description}
    \item[$\Def_X$] Deformations of $X$,
    \item[$\Def_W$] Deformations of $W$,
    \item[$\Def_f$] Deformations of the map $f$ possibly varying both $X$ and $W$,
    \item[$\Def_{f/W}$] Deformations of the map $f$ preserving $W$
        \item[$\Def_{X\backslash f}$] Deformations of the map $f$ preserving $X$
        \item[$\Def_{X\backslash f/W}$] Deformations of the map $f$ preserving $W$ and $X$
\end{description}
The corresponding tangent cohomology groups (or sheaves) are intertwined in the cotangent braid of Buchweitz (compare \cite[p.~446]{GLS}), shown in  Figure \ref{fig: cotangent braid}.
\begin{figure}\caption{The cotangent braid of Buchweitz}\label{fig: cotangent braid}
\begin{equation}
 \label{eq: cotangent braid}
\begin{tikzcd}[row sep =  tiny]
 T^0_{X\backslash f} \arrow [bend right]{dd}[swap]{\mathbf{2}} \arrow {dr}{\mathbf{3}} 
  & &T^0_{f/W}\arrow{dl}[swap]{\mathbf {1}}\arrow[dashed, bend left]{dd}{\mathbf{4}}\\
 & T^0_f \arrow{dr}{\mathbf{3}} \arrow{dl}[swap]{\mathbf{1}} & \\
         T_W^0 \arrow{dr}{\mathbf{2}} \arrow[bend right]{dd}[swap]{\mathbf{1}}& & T^0_X\arrow[dashed]{dl}[swap]{\mathbf{4}} \arrow[bend left]{dd}{\mathbf{3}} \\
         & T^1_{X\backslash f/W}\arrow{dr}{\mathbf{2}} \arrow[dashed]{dl}[swap]{\mathbf{4}} & \\
          T^1_{f/W}\arrow[bend right, dashed]{dd}[swap]{\mathbf{4}} \arrow{dr}{\mathbf{1}} & & T^1_{X\backslash f} \arrow[bend left]{dd}{\mathbf{2}} \arrow{dl}[swap]{\mathbf{3}} \\
         & T^1_f \arrow{dr}{\mathbf{1}} \arrow{dl}[swap]{\mathbf{3}} & \\
       T^1_X \arrow[dashed]{dr}{\mathbf{4}} \arrow[bend right]{dd}[swap]{\mathbf{3}} & & T^1_W \arrow{dl}[swap]{\mathbf{2}} \arrow[bend left]{dd}{\mathbf{1}}[swap]{\alpha} \\
         & T^2_{X\backslash f/W} \arrow[dashed]{dr}{\mathbf{4}} \arrow{dl}[swap]{\mathbf{2}} & \\
         T^2_{X\backslash f} \arrow [bend right]{dd}[swap]{\mathbf{2}} \arrow {dr}{\mathbf{3}} 
  & &T^2_{f/W}\arrow{dl}[swap]{\mathbf {1}}\arrow[dashed, bend left]{dd}{\mathbf{4}}\\
 & T^2_f \arrow{dr}{\mathbf{3}} \arrow{dl}[swap]{\mathbf{1}} & \\
         T_W^2 \arrow{dr}{\mathbf{2}} \arrow[bend right]{dd}[swap]{\mathbf{1}}& & T^2_X\arrow[dashed]{dl}[swap]{\mathbf{4}} \arrow[bend left]{dd}{\mathbf{3}} \\
         & \dots& \\
                  \ldots & & \ldots \\
\end{tikzcd}
\end{equation}
 
\end{figure}

\begin{rem}
 All spaces that we consider are either smooth or local complete intersections, so that for deformation purposes it would be enough to work with the sheaf of K\"ahler differentials. But the deformations of $f$ without fixing source or target need another caliber of theory, which is the reason why we use the general machinery. 
 
 Very similar problems have been considered previously  for example in \cite{CvS06, MR2672280, MR1422433}, but we find it most transparent to start from scratch. 
\end{rem}

\begin{lem}\label{lem: deformation preserves map}
 Assume that $W$ is smooth and that
 \begin{enumerate}
  \item $H^0(W, \kt_W)\to H^0(W, f_*f^*\kt_W)$ is an isomorphism,
  \item $H^1(W, \kt_W)\to H^1(W, f_*f^*\kt_W)$ is an isomorphism,
  \item $H^2(W, \kt_W)=0$.
 \end{enumerate}
Then the natural forgetful maps $T_f^i\to T^i _X$ are isomorphisms for $i\leq 2$ and $\Def_f\isom \Def_X$. In other words,  every deformation of $X$ is induced by a deformation of the map $f$ in a unique way. 
\end{lem}
\begin{proof}
From sequence $\mathbf3$ in \eqref{eq: cotangent braid} we see that it is enough to show that $T^i_{X \backslash f} = 0$ for $i = 0,1,2$. Using sequence $\mathbf2$ from \eqref{eq: cotangent braid} and the isomorphisms $T_W^i = H^i(W, \kt_W)$  and $T^i_{X
 \backslash f / W} = H^{i-1}(W, f_*f^*\ko_W)$ from  \cite[Prop. 3.4.2]{Sernesi}, the result follows. 
\end{proof}

Now we work out more specifically some groups and maps in the cotangent braid in the case of a double cover $f\colon X \to W$ branched over a divisor $B$. We will always assume that $W$ is smooth and now recall the standard theory from \cite[I.17]{BHPV}: we have   $f_* \ko_X = \ko_W \oplus \inverse \kl$ for some line bundle $\kl$ such that $\kl^{\tensor 2} \isom \ko_W(B)$.
Consider the geometric line bundle $\pi \colon |\kl|=\underline{\mathsf{Spec}}_W\left( \Sym^\bullet \inverse \kl\right) \to W$ and denote the tautological section in $\pi^*\kl$ by $z$. If $\sigma_B$ is the section defining the branch locus, then $X$ is the divisor in $|\kl|$ defined by the section $\sigma_X = z^2 - \pi^* \sigma_B$ of $\pi^* \ko_W(B)$. In particular, as a hypersurface in a smooth variety, $X$ is a local complete intersection.

On $X$ the section $\sigma_R = z$ is a square root of $\pi^*\sigma_B$ and as such defines the ramification divisor $R$.

\begin{prop}\label{prop: identifying alpha}
 Let $f \colon X \to W$ be a double cover of a smooth variety. In the notation above the following hold:
 \begin{enumerate}
  \item $\kt_{f/W}^i = 0 $ for $i \neq 1$ and $\kt_{f/W}^1 \isom \ko_X(f^* B)|_R\isom \ko_B(B)$ (under the identification $f\colon R \isom B$).
  \item If $h^1(\ko_W) = 0 $ then 
  \[ T^i_{f/W} = H^{i-1} (X, \kt_{f/W}^1) = 
 \begin{cases}
0 & i \neq 1, 2\\
H^0(X, \ko_X(B))/\langle \sigma_B\rangle & i = 1\\
H^1(B, \ko_B(B)) & i = 2
 \end{cases},
 \]
In other words, infinitesimal deformations of the map with fixed target $W$ are exactly given by deformations of the branch divisor. 
\item The map $T^1_W \to T^2_{f/W}$ in sequence  $\mathbf4$ in the cotangent braid \eqref{eq: cotangent braid} is under the identification
\begin{equation}\label{eq: connecting map}\alpha \colon  H^1(\kt_W) = T^1_W \to T^2_{f/W} = H^1(B, \ko_B(B))\end{equation}
induced by the map of sheaves $\alpha \colon \kt_W \to \ko_B(B)$ described in the following way: for a local vector field $\xi$ the section $\alpha(\xi)$ is the restriction to $B$ of the derivative of the equation of $B$ in direction $\xi$. That is,
\[ \alpha(\xi) = \left(\xi \lrcorner d\sigma_B\right)|_B = \xi(\sigma_B)|_B.\]
 \end{enumerate}
\end{prop}
\begin{proof}
Note that the sheaf of relative differentials of the line bundle $\pi \colon |\kl| \to W$ is naturally $\omega_{|\kl|/W} = \inverse\kl$. Thus we can consider the following commutative diagram
 \begin{equation}\label{eq: diag for seq 4}
\begin{tikzcd}
{}& & 0 \dar & 0 \dar \\
&& f^*\Omega_W\dar \rar[equal] & f^*\Omega_W\dar\\
0 \rar &\kn^*_{X/|\kl|}\dar[equal] \rar{\cdot d \sigma_X}&\Omega_{|\kl|}|_X\dar \rar &  \Omega_X \rar\dar &0 \\
0 \rar & f^*\ko_W(-B) \rar{\cdot 2z} & f^* \inverse \kl\rar\dar & \Omega_{X/W} \rar\dar & 0\\
&& 0 & 0 
\end{tikzcd}, 
 \end{equation}
where the second row is the usual conormal sequence for a divisor and the first map is given by multiplication with the derivative of the section $\sigma_X$ as indicated.
In the third row we only look at the differential in fibre direction, therefore under the identification $\omega_{|\kl|/W} = \inverse\kl$ the first map becomes multiplication with $2z = 2\sigma_R$.

Note that because $X$ is a local complete intersection and because $\IL_{X\backslash f \slash W} = f^*\Omega_W$ as in the proof of Lemma \ref{lem: deformation preserves map}, applying $\Hom( - , \ko_X)$ to the third column of \eqref{eq: diag for seq 4} gives sequence $\mathbf4$ in the cotangent braid \eqref{eq: cotangent braid}.

Applying $\shom_{\ko_X}(- , \ko_X)$ to the last row, we get a short exact sequence
\begin{equation}\label{eq: t^1_X/W}
 \begin{tikzcd}
0 \rar & f^* \kl \rar{\cdot \sigma_R} & \ko_X(f^*B) \rar & \kt_{f/W}^1 \rar & 0
 \end{tikzcd},
\end{equation}
so $\kt_{f/W}^i = 0 $ for $i \neq 1$ and $\kt_{f/W}^1 \isom \ko_X(f^* B)|_R\isom \ko_B(B)$ (under the identification $f\colon R \isom B$).

If we  push forward to $W$ and then take cohomology, then multiplication with the equation of the ramification divisor on $X$ exchanges the invariant and anti-invariant subspaces of cohomology groups, so we can identify
\[\begin{tikzcd}[ampersand replacement = \& ]
H^i(X, f^*\kl)\rar{\sigma_R} \dar[equal] \& H^i(X, f^*\ko_W(B))\dar[equal]\\
H^i(W, \kl) \oplus H^i(W, \ko_W)\rar{\begin{pmatrix}0 & \sigma_B\\ 1 & 0 \end{pmatrix}}\& H^i(W, \ko_W(B)) \oplus H^i(W, \inverse\kl(B))
\end{tikzcd} .
\]
Canceling the components on which the map is an isomorphism and noting that $H^1(\ko_W) = 0$ by assumption, the long exact sequence associated to \eqref{eq: t^1_X/W} and the local-to-global Ext sequence give the claim (ii).

To identify the map $\alpha$ in sequence 1 in \eqref{eq: cotangent braid} we apply $\shom_{\ko_X}(- \ko_X)$ to \eqref{eq: diag for seq 4} and obtain
\[
\begin{tikzcd}
&&& 0 & \kt^1_{X/W}\\
 &0&& f^* \kt_W \uar\rar[equal] & f^*\kt_W \uar{\alpha}\\
  0 & \kt^1_X\uar \lar & f^* \ko_W(B)\lar & \kt_{|\kl|}|_X\lar[swap] {-\lrcorner d\sigma_X} \uar& \kt_X \lar\uar & 0 \lar\\
  0 & \kt^1_{X/W} \uar\lar & f^* \ko_W(B)\lar\uar[equal] & f^*\kl\lar \uar& 0 \lar\uar\\
  & f^*\kt_W\uar{\alpha}&& 0 \uar
\end{tikzcd}.
\]
Chasing through the diagram shows that the map $\alpha$ is defined as follows: given a (local) vector field $\xi$ on $W$ we can choose any lift $\tilde \xi$ to a vector field on $|\kl|$. Then $\alpha(\xi) = \tilde \xi \lrcorner d\sigma_X$ projected to $\kt^1_{X/W}$, where $\sigma_X$ is the equation defining $X$ and $\lrcorner$ is the contraction of $1$-forms with vector fields. This gives the claimed map.
\end{proof}

\begin{rem}
 To summarize, we get that the sequence $\mathbf1$ in the cotangent braid in the situation of a double cover is associated to dual of the residue sequence $ 0\to \Omega_W \to \Omega_W(\log B) \to \ko_B\to 0$, which is also explained in \cite{CvS06}.
\end{rem}

\subsection{Cohomology computations}
We will now apply the results of the previous section to standard stable Horikawa surfaces. 
\begin{prop}\label{prop: def preserves map for Horikawa surfaces}
Let $W = \IF_{m}$ be a Hirzebruch surface and  $f\colon X\to W$ be a double cover branched over any divisor $B$ in $|6\sigma_\infty + 2a\Gamma|$ with $a>2m+2$. Then the natural map $\Def_f \to \Def_X$ is an isomorphism.
 \end{prop}
\begin{proof}
 We need to check the conditions of Lemma \ref{lem: deformation preserves map}. First note that $H^2(W, \kt_W) = 0 $ by e.g. \cite[Appendix B]{Sernesi} or the computations done below.
 For the other two conditions we follow the proof of \cite[Lem 2.3]{horikawa1}.

 First note that since $f_*f^*\kt_W = \kt_W\tensor f_*\ko_X = \kt_W\oplus \kt_W(-3\sigma_\infty - a\Gamma)$, it is enough to show that $H^i(\kt_W(-3\sigma_\infty - a\Gamma))=0$ for $i = 0,1$. 
 
 In the relative tangent sequence for the fibration $\pi\colon \IF_m \to \IP^1$,
 \[ 0 \to \kt_{\IF_m/\IP^1}\to \kt_{\IF_m} \to \pi^*\kt_{\IP^1}\to 0, \]
we can identify $\pi^*\kt_{\IP^1} = \ko_{\IF_m}(2\Gamma)$ and $\kt_{\IF_m/\IP^1} = \ko_{\IF_m}(2\sigma_\infty + m \Gamma)$.
Twisting with $(-3\sigma_\infty-a \Gamma)$ we immediately get 
\begin{gather*}
 h^0( -\sigma_\infty + (m-a)\Gamma) = 0\\ h^1( -\sigma_\infty + (m-a)\Gamma)=0,\\
 h^0( -3\sigma_\infty + (2-a)\Gamma)=0.\\
  \end{gather*}
 Furthermore,
 \begin{align*}
   h^1( -3\sigma_\infty + (2-a)\Gamma)& = h^1(K_W - \sigma_\infty + (4+m-a)\Gamma) \\
   &= h^1(\sigma_\infty +(a-4-m) \Gamma) \\
   &= h^1(\ko_{\IP^1}(a-4-m)) + h^1(\ko_{\IP^1}(a-4-2m))\\
   &=0,
 \end{align*}
because by assumption $a >  2m+2$, i.e. $a-4-2m > -2$. The required vanishing follows from the long exact cohomology sequence. 
    \end{proof}

In the situation of standard Horikawa surfaces we want to compute the map \eqref{eq: connecting map} explicitly, thereby taking control over sequence $\mathbf1$ in the cotangent braid \eqref{eq: cotangent braid}. 

Fix $m\geq 0$ and consider $B \in |6 \sigma_\infty + 2a \Gamma|$ with $a>2m+2$ on $W = \IF_m$. 
\begin{lem}\label{lem: restriction to infinity}
 If $B = k \sigma_\infty + B'$ with $k \in \{ 0 , 1, 2\}$ then the inclusions $k\sigma_\infty \into B \into W$ induce isomorphisms
 \[ H^1(W, \ko_W(B)) \isom H^1(B, \ko_B(B))\isom  H^1(k\sigma_\infty, \ko_{k\sigma_\infty}(B))\]
 and the dimension of this group is
\[
   h^1(W, \ko_W(6\sigma_\infty + 2a\Gamma))  =
   \begin{cases}
    0 & 2a \geq 6m -1\\
    6m-2a-1 & 6m -1 > 2a \geq 5m -1\\
    11m -4a -2 &  5m-1>2a > 4m+4
       \end{cases}.
\]
 Under the given conditions the group vanishes if $m\leq 4$. 
\end{lem}
\begin{proof}
 The first isomorphism follows from the restriction sequence and $H^1(\ko_W) = H^2(\ko_W) = 0 $. Using $2a>4m +4$ dimensions are computed as
 \begin{align*}
  h^1(W, \ko_W(6\sigma_\infty + 2a\Gamma)) & = \sum_{i=0}^6 h^1(\IP^1, \ko_{\IP^1}(2a - im) \\
   &=  h^1(\IP^1, \ko_{\IP^1}(2a - 5m)+  h^1(\IP^1, \ko_{\IP^1}(2a - 6m)\\
   & =
   \begin{cases}
    0 & 2a \geq 6m -1\\
    6m-2a-1 & 6m -1 > 2a \geq 5m -1\\
    11m -4a -2 &  5m-1>2a > 4m+4
       \end{cases}
 \end{align*}

 Now assume that $B = k \sigma_\infty + B'$ with $k = 1$ or $k =  2$. Then there is an exact sequence
 \[ 0 \to \ko_{B'}(B') \to \ko_B(B) \to \ko_{k\sigma_\infty}(B) \to 0.\]
 The corresponding cohomology sequence gives an isomorphism $H^1(B, \ko_B(B)) \isom H^1(k\sigma_\infty, \ko_{k\sigma_\infty}(B))$.

  Note that Lemma \ref{lem: standard} ensures that $B$ contains $\sigma_\infty$ or even $2\sigma_\infty$ when  the co\-ho\-mo\-logy group $H^1(W, \ko_W(B))$ is non-zero. 
 \end{proof}

 To compute the map \eqref{eq: connecting map} explicitly, we want explicit \v Cech cohomology descriptions of the relevant groups. 
  We will assume that $m>0$ since otherwise the cohomology groups we are interested in vanish anyway. We closely follow \cite[Appendix B]{Sernesi}. Let us set up our notation  starting from the toric model of the Hirzebruch surface given by 
 \[ \mat{ t_0 & t_1 & x_0 & x_1 \\ 1& 1& 0 & -m \\ 0 & 0& 1 & 1& }\]
 with projection $\pi \colon \IF_m \to \IP^1$.
Let
\begin{equation}\label{eq: coordinate transforms}
\begin{split}
\tau  = \frac {t_0}{t_1} &\qquad \xi = \frac{x_1t_1^m}{x_0}\\
\tau'  = \inverse \tau  & \qquad \xi'= \frac{x_1t_0^m}{x_0} = \xi\tau^m
\end{split}
\end{equation}
and $R = \spec[\tau, \inverse \tau]$, considered as a $\IZ$-graded ring. 
Then with $U = \{t_1\neq 0 \}$ and $U '  = \{t_0\neq 0 \}$ we have $\IP^1 = U \cup U'$. Consider
\begin{align*} & V = \inverse \pi (U) \supset V_0 : = \spec \IC[\tau, \xi], \\
& V' =  \inverse \pi (U') \supset V'_0 : = \spec \IC[\tau', \xi']
 \end{align*}

We will compute \v Cech cohomology with respect to the covering $W = V \cup V'$ but represent sections by their restrictions to the affine subsets $V_0$ respectively $V_0'$. On $V_0$ the curve $B$ is defined by an equation 
\[ \sigma_B^0 = g_0 + g_1\xi + g_2\xi^2+ \dots + g_6 \xi^6\]
for certain $g_i\in \IC[\tau]$. We may assume for simplicity that we see all zeros of the coefficients of $\sigma_B$ in $V\cap V'$, that is, $\deg g_i = 2a -6(m-i)$ and $g_i$ has a non-zero constant term unless it vanishes identically. 
Note that $\deg g_0 = 2a-6m = B.\sigma_\infty$.

\begin{lem}\label{lem: check cohomology of B}
For any $B\in |6\sigma_\infty + 2a\Gamma|$ with $2a> 4m+4$, we have with respect to the above covering
\[\check H^1(W, \ko_W(B))
= \frac1{\sigma_B^0}  \left (1\frac{  R}{ R_{\geq 0} +R_{\leq 2a-6m}} + \xi \frac{  R}{R_{\geq 0} +R_{\leq 2a-5m}}\right).
\]
\end{lem}
\begin{proof}
To compute \v Cech cohomology, we need to describe the section on $V$ and $V'$ and then compare them over $V\cap V'$. For this note that,  by the relation between the coordinates, we  can describe the equation for $B$ on $V'_0$ as
\[ {\sigma_B^0}' = \sigma_B^0\cdot \tau ^{6m-2a}\]
so that by \eqref{eq: coordinate transforms}
\[ \Gamma(V', \ko_W(B)) 
= \frac1{{\sigma_B^0}'} \IC[\tau '] \langle 1, \dots, {\xi'}^6\rangle
= \frac{\tau^{2a-6m}}{{\sigma_B^0}} R_{\leq0} \langle 1, \tau^m \xi , \dots, {\tau^m \xi}^6\rangle.
\]
Then using $2a>4m+4$ we compute
\begin{align*} \check H^1(W, \ko_W(B))
& = \Gamma(V\cap V', \ko_W(B))/\left(\Gamma(V, \ko_W(B))+\Gamma( V', \ko_W(B))\right)\\
 & = \frac1{\sigma_B^0}  \left (1\cdot  R/ \left( R_{\geq 0} +R_{\leq 2a-6m}\right) + \xi \cdot  R/ \left( R_{\geq 0} +R_{\leq 2a-5m}\right)\right)
\end{align*}
as claimed.
\end{proof}

\begin{lem}\label{lem: check cohomology tangent}
 This covering computes $H^1(W, \kt_W)$ and explicitly we have 
 \[ \check H^1(W, \kt_W) = \frac{R\cdot \frac{\del}{\del\xi}}{ R_{\geq 0 } \cdot \frac{\del}{\del\xi} + R_{\leq 0}\cdot \tau^{-m}\frac{\del}{\del\xi}} \isom  \left\langle \tau^{-1}\frac{\del}{\del\xi} , \dots , \tau^{-(m-1)}\frac{\del}{\del\xi} \right\rangle 
,\]
a vector space of dimension $m-1$.
\end{lem}
\begin{proof}
The sheaf $\kt_W$ does not have higher cohomology on $V$ and $V'$, so we can use this covering to compute cohomology. 
Representing again everything in the local coordinates given on $V_0$ and $V_0'$, we get
\begin{align*}
 \Gamma(V, \kt_W) & = \IC[\tau]  \left\langle\frac{\del}{\del \tau} , \frac{\del}{\del\xi}, \xi \frac{\del}{\del\xi}, \xi^2 \frac{\del}{\del\xi}\right\rangle\\
\Gamma(V', \kt_W) & = \IC[\tau']  \left\langle\frac{\del}{\del \tau'} , \frac{\del}{\del\xi'}, \xi \frac{\del}{\del\xi'}, {\xi'}^2 \frac{\del}{\del\xi'}\right\rangle\\
 & = \IC[\inverse \tau] \left\langle  - \tau^2 \frac{\del}{\del \tau} + m\tau \xi\frac{\del}{\del\xi}, \tau^{-m} \frac{\del}{\del \xi}, \xi \frac{\del}{\del \xi}, \tau^{m} \xi^2\frac{\del}{\del \xi}\right\rangle\\
 &\subset \Gamma(V\cap V', \kt_W)
 ,
\end{align*}
where we get the formulas for the coordinate change by differentiating \eqref{eq: coordinate transforms}. It is straightforward to see that we get the basis of the quotient stated above.\footnote{In \cite[Appendix B]{Sernesi} the result of this computation is stated with a missing inverse.}
\end{proof}

\begin{prop}\label{prop: alpha for sigma}
If $B= \sigma_\infty + B'$ and $\sigma_\infty \not < B'$, then the map \eqref{eq: connecting map} is surjective, and it is an isomorphism if and only if $2a = 5m$.
\end{prop}
\begin{proof}
We have $2a \geq 5m$ by Lemma \ref{lem: standard}.
 We write the equation 
 \[ \sigma_B^0 = \xi g_{1} + \text{higher order terms in $\xi$}\]
for some polynomial $g_1$ of degree $2a-5m\geq 0$.
Then we compute the map  \eqref{eq: connecting map} using Lemma \ref{lem: restriction to infinity}
\[
 \check H^1(W, \kt_W) \to \check H^1(\sigma_\infty, \ko_{\sigma_\infty}(B)),
  \quad \tau^{-i} \frac{\del}{\del \xi} \mapsto
 g_1\cdot\tau^{-i} \frac{1}{\sigma_{B}^0}
\]
Since we have chosen coordinates such that $g_1$ has a constant term, the claim follows from Lemma \ref{lem: check cohomology of B}, Lemma \ref{lem: restriction to infinity} and Lemma \ref{lem: check cohomology tangent}. 
\end{proof}

\begin{prop}\label{prop: alpha for 2sigma}
If $B= 2\sigma_\infty + B'$ and $\sigma_\infty \not < B'$, then the map $\alpha$ from \eqref{eq: connecting map}
\begin{enumerate}
 \item is surjective if and only if $H^1(B, \ko_B(B)) = 0 $ if and only if $2a >6m$;
 \item is zero if and only if $2a\geq 5m$;
 \item has rank $5m-2a-1<m-4$ for $4m+4<2a< 5m$.
\end{enumerate}
\end{prop}
Note that by Remark \ref{rem: which cases} the last case can only occur for $m\geq 7$,  so that the inequalities make sense. 

\begin{proof}
 We have $2a > 4m+4$ by Lemma \ref{lem: standard}.
 We write the equation 
 \[ \sigma_B^0 = \xi^2 g_{2} + \text{higher order terms in $\xi$}\]
for some polynomial $g_1$ of degree $2a-4m>4$. By Lemma \ref{lem: restriction to infinity} we can compute the composition 
\[
\begin{tikzcd}\check H^1(W, \kt_W) \rar{\alpha}& H^1(B, \ko_B(B)) \rar & H^1(2\sigma_\infty, \ko_{2\sigma_\infty}(B))
 \end{tikzcd}
\]
and the last group is simply the restriction of the one computed in Lemma \ref{lem: check cohomology of B} to $2\sigma_\infty$. 

For the explicit representatives computed in Lemma \ref{lem: check cohomology tangent} we get
\[ \alpha \left(\tau ^{-i} \frac\del{\del \xi}\right)=  \tau^{-i} \cdot 2\xi \cdot g_2\frac {1}{\sigma_B^0}|_{2\sigma_\infty} \]
Since we have chosen coordinates such that $g_2$ has a constant term, the map surjects on the subspace
\[
\begin{tikzcd}
 \check H^1(W, \kt_W) \arrow[->>]{r} & \xi \frac{  R}{R_{\geq 0} +R_{\leq 2a-5m}} \rar[hookrightarrow] & \check H^1(2\sigma_\infty, \ko_{2\sigma_\infty}(B))
\end{tikzcd}.
\]
Thus the map $\alpha$ is
surjective if and only if $H^1(B, \ko_B(B)) = 0 $ if and only if $2a >6m$. 
The other two items follow by counting the dimension of the spaces involved. 
\end{proof}

\subsection{Deformation-theoretic interpretations}
 We now interpret the computations of the previous section in our context.
 First recall that for a Hirzebruch surface $W = \IF_m$ with $m>0$ we have
 \[ h^0(\kt_W) = m+5, \quad h^1(\kt_W) = m-1, \quad h^2(\kt_W) = 0,\]
 deformations are unobstructed and $W$ deforms to $\IF_{m-2k}$ for $k = 1, \dots , \lfloor m/2\rfloor$. This is explained for example in  \cite[p. 10]{Catanesesuperficial} and a more precise description of the stratification of the universal deformation space has been given by Suwa \cite{suwa73}.

Now for a standard stable Horikawa surface consider sequence $\mathbf1$ in the cotangent braid \eqref{eq: cotangent braid}, and its more worldly incarnation derived from  Proposition \ref{prop: def preserves map for Horikawa surfaces} and Proposition \ref{prop: identifying alpha}:
 {\small
 \begin{equation}
\label{eq: seq 1}
 \begin{tikzcd}[row sep = small, column sep = small]
    0 \rar & T^0_W \dar[equal] \rar
    & T^1_{f/W} 
    \dar[equal] \rar& T_f^1 \dar[equal] \rar& 
T^1_W 
     \dar[equal] \rar{\alpha}&
     T^2_{f/W} 
     \dar[equal] \rar& T_f^2 \dar[equal] \rar&
T^2_W 
     \dar[equal]
\\
  0\rar  & H^0(\kt_W) \rar & H^0(\ko_B(B)) \rar & T_X^1 \rar & H^1(\kt_W) \rar & H^1(\ko_B(B)) \rar &  T_X^2 \rar & 0 
   \end{tikzcd}
\end{equation}
 }
 
 The first part of the the following result reproves some of the results in \cite{horikawa1}. Part of it can also be deduced directly from the concrete families constructed in Section \ref{sect: conection 1} and Section \ref{sect: conection 2}.
 \begin{prop}\label{prop: deformation information}
 Let $f\colon X \to W$ be a standard stable Horikawa surface of type $(m)$ with branch curve $B\in |6\sigma_\infty + 2a \Gamma|$. 
 \begin{enumerate}
  \item If  $2a\geq 6m$,  then deformations of $X$ are unobstructed and every deformation of $W$ can be lifted to a deformation of $X$. In particular $X$ deforms to a Horikawa surface of type $(0)$.
  \item If $2\sigma_\infty\not <B$, i.e. $2a\geq 5m$,  then deformations of $X$ are unobstructed. If $2a <5m$ then $X$ deforms to a Horikawa surface of lower type.
  \item Assume $2a = 5m$ and write $B = k\sigma_\infty + B'$.
  \begin{enumerate}
   \item If $k= 1$ then deformations of $X$ are unobstructed and $\Def_X \isom \Def_{f/\IF_m}$; that is, every small deformation is of type $(m)$ again.
   \item If $k = 2$, then $\dim T^1_X = 12a-15m -1$.
  \end{enumerate}
\item If $5m>2a>4m+4$ and $B = 2\sigma_\infty+B'$, then the map $T^1_X \to T_W^1$ is non-trivial of rank $K_X^2/2 - m +3 = 2a -4m -1 \geq 4$.
  \end{enumerate}
 \end{prop}
\begin{proof}
 All items follow immediately from \eqref{eq: seq 1} in conjunction with Proposition \ref{prop: alpha for sigma} and Proposition \ref{prop: alpha for 2sigma}, where we use that a deformation functor is unobstructed if $T^2=0$ and our knowledge of deformations of Hirzebruch surfaces.
\end{proof}

We can now complete the proof of Theorem A, giving information on the local structure $\bar\gothH_{8k}$ at a surface as described in Example \ref{exam: special}.
\begin{cor}\label{cor: Thm A local structure}
Suppose that $K_X^2 = 8\ell$ with $\ell >2$ and consider a surface $X$ representing a general  point in $\gothD \subset \bar \gothH^{\mathrm I} \cap \bar\gothH^{\mathrm{II}}$ in the intersection of the special and general component.  Then the point $[X]$ representing $X$ in $\bar\gothH_{k\ell}$ lies in exactly these two irreducible components,  but $\bar\gothH_{8\ell}$ is not normal crossing near $[X]$. 
\end{cor}
\begin{proof}
 We know that every deformation of $X$ is a again a standard stable Horikawa surface (Proposition \ref{prop: def preserves map for Horikawa surfaces}) and that the type can only decrease, so $X$ can only deform to one of these two components. 
  If the moduli space were normal crossing at $[X]$, then $\dim T_{[X]}\bar\gothH_{k\ell} = \dim \gothH^\mathrm{I}_{k\ell}+1$, but by Proposition \ref{prop: deformation information} and Theorem \ref{thm: classical stuff} we compute 
  \[\dim T_{[X]}\bar\gothH_{k\ell}-\dim\gothH^\mathrm{I}_{k\ell} = \dim T^1_X - \dim\gothH^\mathrm{I}_{k\ell} + 1 + (m-2)>1.\]
 This proves the claim. 
\end{proof}

\subsection{Proof of Theorem B}
\label{sect: proof B}
We are now ready to prove Theorem B from the introduction.

By induction and Propositions \ref{prop: connect to previous even} and \ref{prop: connect to previous odd} we know that all subsets $\Hbar{m}_{2k}$ are contained in the connected component of the moduli space containing $\gothH_{2k}$. 

So let us now consider a general surface $X$ parametrised by $\Hbar{m}_{2k}$ for $2k>2m>k+4$. In particular, we may assume that the reduction of $\Hbar{m}_{2k}$ is smooth at $[X]$.

Then by Lemma \ref{lem: standard}, the surface $X$ is not classical but a  non-normal stable surface and by Proposition \ref{prop: def preserves map for Horikawa surfaces} every deformation of $X$ lifts to a deformation of the map $f\colon X \to \IF_m$. 
By the deformation theory of Hirzebruch surfaces, the type cannot increase in a neighbourhood of $X$ in $\bar\gothH_{2k}$, so all deformations of $X$ are contained in the union $\bigcup_{m'\leq m}\Hbar{m'}_{2k}$. But on the other hand, the union of all strata of smaller type has smaller dimension than $\Hbar{m}_{2k}$ by Corollary \ref{cor: increasing dimension trailing.} so that the general $X$ cannot lie in the closure of these components. In other words, all non-infinitesimal deformations of $X$ are again of type $(m$) and $\Hbar{m}_{2k}$ forms (an open subset of) an irreducible component of the moduli space. 

Endowing $\Hbar{m}_{2k}$ with the scheme structure defined by the moduli space of stable surfaces and noting that $X$ is Gorenstein, so every deformation is admissible, we show that the component is generically non-reduced by computing the dimension of the tangent space at the general point $X$:
\[ \dim T_{[X]} \Hbar{m}_{2k}=\dim T^1_X = \dim\Hbar{m}_{2k} + k-m+3> \dim\Hbar{m}_{2k}= \dim T_{[X]} \left(\Hbar{m}_{2k}\right)_\text{red}\]
by Proposition \ref{prop: deformation information} $(iv)$.
This concludes the proof. \hfill \qed

\bibliographystyle{plain}
\bibliography{references}

\end{document}